\documentclass[11pt]{amsart} 
\usepackage{amsmath,amsthm}
\usepackage{xypic}
\usepackage{enumerate}
\usepackage{lscape}

\usepackage{amscd}
\usepackage{amsfonts,amssymb,latexsym} 
\usepackage[all]{xy} 
\xyoption{arc}
\CompileMatrices

\newcommand{\udots}{\mathinner{\mskip1mu\raise1pt\vbox{\kern7pt\hbox{.}}
\mskip2mu\raise4pt\hbox{.}\mskip2mu\raise7pt\hbox{.}\mskip1mu}}

\newcommand{\SD}{{\mathcal{D}}}

\newcommand{\SM}{{\mathcal{M}}}

\newcommand{\SO}{{\mathcal{O}}}

\newcommand{\Gr}{\operatorname{Gr}}
\newcommand{\Spec}{\operatorname{Spec}}

\newcommand{\Sym}{\operatorname{Sym}}
\newcommand{\id}{\operatorname{Id}}

\newcommand{\End}{\operatorname{End}}

\newcommand{\tr}{\operatorname{tr}}

\newcommand{\ad}{\operatorname{ad}}
\newcommand{\smb}{\operatorname{sb}}

\newcommand{\op}{\operatorname}

\newtheorem{proposition}{Proposition}[section]
\newtheorem{theorem}[proposition]{Theorem}
\newtheorem{definition}[proposition]{Definition}

\newtheorem{lemma}[proposition]{Lemma}

\newtheorem{corollary}[proposition]{Corollary}
\newtheorem{remark}[proposition]{Remark}

\numberwithin{equation}{section}

\title[Hitchin map for $\Lambda$-modules positive characteristic]{Hitchin map for the moduli space of $\Lambda$-modules in positive characteristic}
\author[D. Alfaya]{David Alfaya}
\author[C. Pauly]{Christian Pauly}

\date{}

\address{Department of Applied Mathematics and Institute for Research in Technology, ICAI School of Engineering,
Comillas Pontifical University, C/Alberto Aguilera 25, 28015 Madrid, Spain}

\email{dalfaya@comillas.edu}

\address{Laboratoire J.-A. Dieudonn\'e, Universit\'e Côte d'Azur,
Parc Valrose, 06108 Nice Cedex 02, France}

\email{pauly@unice.fr}


\begin{document}

\begin{abstract}
Building on Simpson's original definition over the complex numbers, we introduce the notion of restricted sheaf $\Lambda$ of 
rings of differential operators on a variety defined over a field of positive characteristic. We define the notion of $p$-curvature for
$\Lambda$-modules and the analogue of the Hitchin map on the moduli space of $\Lambda$-modules. We show that under certain conditions
this Hitchin map descends under the Frobenius map of the underlying variety and we give examples.
\end{abstract}

\maketitle

\section{Introduction}

The notion of sheaf of rings of differential operators $\Lambda$ over a smooth variety $X$ defined over an algebraically closed field 
$\mathbb{K}$ and the associated notion of $\Lambda$-module for $\SO_X$-modules over $X$ was introduced in \cite{Simpson1} over the complex
numbers $\mathbb{K} = \mathbb{C}$ as a way to give a unifying structure for $\mathcal{D}_X$-modules, i.e. vector bundles with an integrable 
connection, and Higgs sheaves over $X$. Other examples of $\Lambda$-modules include connections along a foliation or logarithmic connections.

In this paper we consider Simpson's original definition of sheaf of rings of differential operators 
$\Lambda$ over a field $\mathbb{K}$ of characteristic $p > 0$. Note that the sheaf
of rings of crystalline differential operators $\mathcal{D}_X$ (see \cite{BO} or \cite{BMR}) defined as the enveloping algebra of the Lie algebroid
$T_X$ is such a sheaf of rings of differential operators, but the usual sheaf of differential operators 
(e.g. \cite[Section 16]{EGA4}) is not. One of the main
features of the sheaf of rings $\mathcal{D}_X$ in positive characteristic is its large center, which can be described by using the $p$-th
power map, or $p$-structure, on the Lie algebroid $T_X$. Our first contribution to the general study of $\Lambda$-modules in positive characteristic 
is the definition of {\em restricted} sheaf of rings of differential operators (see Definition \ref{defrestrictedlambda})  
obtained by equipping $\Lambda$ with a $p$-structure.  Examples of restricted sheaves of rings of differential
operators already appeared in \cite{Langer} as universal enveloping algebras of restricted Lie algebroids. 
New non-split examples are given, for instance, by the sheaf of rings of twisted differential operators $\mathcal{D}_X(L)$ for some line
bundle $L$ over $X$ (see Subsection \ref{twisteddiffop}).

The main purpose of this paper is to prove a property of the analogue of the Hitchin map for restricted
$\Lambda$-modules in positive characteristic over a projective variety $X$. 
First, we check (Section 5) that the notion of $p$-curvature 
$\psi_\nabla$ of a $\Lambda$-module $E$ over $X$ adapts to our general set-up and thus defines for each
$\Lambda$-module structure on the sheaf $E$
a $F^* H^\vee$-valued Higgs field on $E$, where $H$ is the first quotient $\Lambda_1/ \Lambda_0$
associated to the filtration $\Lambda_0 \subset \Lambda_1 \subset \cdots \subset \Lambda$ and $F$
is the Frobenius map of $X$. Thus by applying the classical Hitchin map to the Higgs field $\psi_\nabla$
we obtain a morphism
$$ h_\Lambda : \mathcal{M}_X^\Lambda(r, P) \rightarrow \mathcal{A}_r(X, F^* H^\vee),$$
where  $\mathcal{M}_X^\Lambda(r, P)$ is the moduli space parameterizing Giesecker semi-stable 
$\Lambda$-modules over $X$ of rank $r$ and with Hilbert polynomial $P$, and $\mathcal{A}_r(X, F^* H^\vee)$
is the Hitchin base for the vector bundle $F^*H^\vee$. Under the assumption that the anchor map
$\delta : \Lambda_1/ \Lambda_0 \rightarrow T_X$ induced by the commutator between elements of $\Lambda_1$ and local regular functions in $\SO_X$ is generically surjective, our main result (Theorem 
\ref{thm:HitchinDescent}) says that the coefficients of the characteristic polynomial of $\psi_\nabla$
descend under the Frobenius map $F$ of the variety $X$. Equivalently, this means that the Hitchin
morphism $h_\Lambda$ factorizes through
\begin{eqnarray} \label{Hitchinmapfact}
h'_\Lambda : \mathcal{M}_X^\Lambda(r, P) \rightarrow \mathcal{A}_r(X, H^\vee),
\end{eqnarray}
followed by the pull-back under the Frobenius map $F$ of global sections. The latter theorem was first proved 
in \cite{LP01} for a smooth projective curve $X$ and for $\Lambda = \mathcal{D}_X$. It was observed in \cite[Section 2.5]{EG} 
that in the case $\Lambda = \mathcal{D}_X$ the proof follows rather directly from the fact that the $p$-curvature
$\psi_\nabla$ is flat for the natural connection on the sheaf $\mathrm{End}(E) \otimes F^* \Omega^1_X$, already proved in
\cite[Proposition 5.2.3]{Katz70}, and moreover their argument is independent of the dimension of the variety $X$.
In this paper we show that the elegant argument given in \cite{EG} can be adapted to general restricted $\Lambda$-modules
under the assumption that the anchor map $\delta: \Lambda_1/ \Lambda_0 \rightarrow T_X$ is generically surjective. 
We also give an example showing that the result is false when $\delta$ is not generically surjective.

In the last section we present an analogue of the main Theorem in a relative situation by taking the Rees construction 
$\Lambda^R$ on $X \times \mathbb{A}^1$ over $\mathbb{A}^1$ obtained from a sheaf of rings of differential operators $\Lambda$ on $X$.
Here we need to restrict attention to sheaves $\Lambda$ obtained as a universal enveloping algebra of a restricted Lie algebroid $H$ over $X$.
Our theorem (Theorem \ref{HitchindescentRees}) then gives an explicit deformation over the affine line $\mathbb{A}^1$ of the classical Hitchin map
of $H^\vee$-valued Higgs sheaves to the Hitchin map \eqref{Hitchinmapfact} $h'_\Lambda$ of $\Lambda$-modules. This result was already 
obtained in \cite{LP01} for a smooth projective curve $X$ in the case where $\Lambda = \mathcal{D}_X$ and $H = T_X$, see also 
\cite[Section 4.5]{Langer} for some partial generalizations.

Finally we mention that the fibers of the Hitchin map \eqref{Hitchinmapfact} $h'_\Lambda$ are described in \cite{Groech16} for a smooth projective 
curve $X$ and for $\Lambda = \mathcal{D}_X$. For general $X$ and $\Lambda$, a description of the fibers of $h'_\Lambda$ seems to be missing in the literature and studying it would be an interesting future line of work.

We would like to thank Carlos Simpson for many useful discussions during the preparation of this article.

\noindent\textbf{Acknowledgments.} 
This work was started during a research stay in 2017 of the first-named author at 
the Laboratoire J.-A. Dieudonn{\'e} at the Universit{\'e} Côte d'Azur and he would like to thank the laboratory for its hospitality. This research was partially funded by MINECO (grants MTM2016-79400-P, PID2019-108936GB-C21 and ICMAT Severo Ochoa project SEV-2015-0554) and the 7th European Union Framework Programme (Marie Curie IRSES grant 612534 project MODULI). During the development of this work, the first-named author was also supported by a predoctoral grant from Fundaci\'on La Caixa -- Severo Ochoa International Ph.D. Program and a postdoctoral position associated to the ICMAT Severo Ochoa project.

\section{Preliminaries on sheaves of rings of differential operators}

\subsection{Definitions and properties}

Let $\mathbb{K}$ be an algebraically closed field. Let $X$ and $S$ be schemes of finite type over $\mathbb{K}$ and let 
$$ \pi : X \longrightarrow S$$ 
be a morphism. We recall from \cite[Section 2]{Simpson1} the definition of 
sheaf of rings of differential operators on 
$X$ over $S$. We note that the original definition in \cite{Simpson1} was given over $\mathbb{K} = \mathbb{C}$, but it can be 
considered over an arbitrary base field $\mathbb{K}$.

\begin{definition}
A sheaf of rings of differential operators on $X$ over $S$ is a sheaf of associative and unital $\SO_S$-algebras $\Lambda$ over $X$ with a filtration
$\Lambda_0 \subset \Lambda_1 \subset \cdots $ which satisfies the properties
\begin{enumerate}
    \item $\Lambda = \bigcup_{i=0}^\infty \Lambda_i$ and $\Lambda_i \cdot \Lambda_j \subset \Lambda_{i+j}$.
    \item The image of $\SO_X \to \Lambda$ equals $\Lambda_0$.
    \item The image of $\pi^{-1} (\SO_S)$ in $\SO_X$ is contained in the center of $\Lambda$.
    \item The left and right $\SO_X$-module structures on $\op{Gr}_i(\Lambda) := \Lambda_i/ \Lambda_{i-1}$ are 
    equal.
    \item The $\SO_X$-modules $\op{Gr}_i(\Lambda)$ are coherent.
    \item The graded $\SO_X$-algebra $\op{Gr}^\bullet(\Lambda) := \bigoplus_{i=0}^\infty \op{Gr}_i(\Lambda)$ is generated by
    $\op{Gr}_1(\Lambda)$.
\end{enumerate}
\end{definition}

Because of property (4) we have that for each $D\in \Lambda_1$ the commutator $[D,f]$ with $f\in \SO_X$ is an element of $\Lambda_0$. 
Moreover, for each $D\in \Lambda_1$ and each $f,g\in \SO_X$ we have
$$[D,fg]=Dfg-fgD=Dfg-fDg+fDg-fgD=[D,f] g +f [D,g].$$
Thus, assuming  that $\Lambda_0 = \SO_X$, we see that the map 
$[D,-]:\SO_X \to \SO_X$ is a $\SO_S$-derivation that we will denote by $\delta_D$ (i.e., $\delta_D(f)=[D,f]$). 
Moreover, let us denote $H=\Lambda_1/\Lambda_0$. Then we have a short exact sequence
\begin{equation} \label{sbmap}
0\longrightarrow \Lambda_0=\SO_X \longrightarrow \Lambda_1 \stackrel{\smb}{\longrightarrow} H 
\longrightarrow 0.
\end{equation}
We call the map $\Lambda_1 \longrightarrow H$ the symbol map and we will denote it by $\smb$. We also note that the 
$\SO_X$-linear map $\delta: D \mapsto \delta_D$ factorizes through $H$, so that we obtain an $\SO_X$-linear map, also
denoted 
$$ \delta: H \longrightarrow \op{Der}_{\SO_S}(\SO_X, \SO_X) = T_{X/S}, $$
called the anchor map. Here $T_{X/S}$ is the relative tangent sheaf.

\bigskip

\begin{remark}
\label{rmk:delta0}
The condition that the anchor map $\delta = 0$ is easily seen to be equivalent to the fact that
the right and left $\SO_X$-module structures on $\Lambda$ are the same.
\end{remark}

\bigskip

In this paper we will be sometimes interested in sheaves of rings of differential operators having more properties.

\begin{definition}
Let $\Lambda$ be a sheaf of rings of differential operators on $X$ over $S$ with $H = \Lambda_1/ \Lambda_0$.
We say that $\Lambda$ is 
\begin{itemize}
    \item  almost abelian, if the graded algebra $\op{Gr}^\bullet(\Lambda)$ is abelian.
    \item  almost polynomial, if $\SO_X = \Lambda_0$, $H$ is locally free and the graded algebra $\op{Gr}^\bullet(\Lambda)$ equals
    the symmetric algebra $\mathrm{Sym}^\bullet(H)$.
    \item  split almost polynomial, if $\Lambda$ is almost polynomial and the exact sequence (\ref{sbmap}) is split.
\end{itemize}
\end{definition}

For completeness we recall the following

\begin{definition}
A $\SO_S$-Lie algebroid on $X$ over $S$ is a triple $(H,[-,-], \delta)$
consisting of an $\SO_X$-module $H$, which is also a sheaf of
$\SO_S$-Lie algebras, and an $\SO_X$-linear anchor map 
$\delta : H \to T_{X/S}$ satisfying the following
condition for all local sections $f \in \SO_X$ and $D_1,D_2 \in H$
$$[D_1,fD_2] = f [D_1 , D_2] + \delta_{D_1}(f)D_2. $$
\end{definition}

\begin{remark}
If $\Lambda$ is almost abelian, then $(H = \Lambda_1/ \Lambda_0 ,[-,-], \delta)$ is a $\SO_S$-Lie algebroid on $X$ (see 
Proposition \ref{HrestrictedLiealgebroid}  for the ``restricted" version).
\end{remark}

\bigskip

\subsection{Restricted sheaf of rings of differential operators}

From now on we assume that the characteristic of $\mathbb{K}$ is $p > 0$. In that situation we introduce the
following

\begin{definition} \label{defrestrictedlambda}
A restricted sheaf of rings of differential operators on $X$ over $S$ is a sheaf of rings of 
differential operators $\Lambda$ on $X$ over $S$ together with a  map
\begin{eqnarray*}
\xymatrixrowsep{0.05pc}
\xymatrixcolsep{0.3pc}
\xymatrix{
[p]&:&\Lambda_1 \ar[rrrr] &&&& \Lambda_1\\
&& D \ar@{|->}[rrrr] &&&& D^{[p]}
}
\end{eqnarray*}
called a $p$-structure, such that for every local sections $D,D_1,D_2\in \Lambda_1$ and every local section 
$f\in \SO_{X}$ the following properties hold
\begin{enumerate}
\item $\ad(D^{[p]})=\ad(D)^p$
\item $(D_1+D_2)^{[p]} = D_1^{[p]}+D_2^{[p]} + \sum_{i=1}^{p-1} s_i(D_1,D_2)$
\item $(fD)^{[p]}=f^pD^{[p]}+\delta_{fD}^{p-1}(f) D$
\item $f^{[p]}=f^p$
\end{enumerate}
where $s_i(x,y)$ are the universal Lie polynomials for the commutator in the associative algebra $\Lambda$, defined by the following expression in $\Lambda[t]$
$$\ad(tx+y)^{p-1}(x)=\sum_{i=1}^{p-1}is_i(x,y)t^{i-1}.$$
\end{definition}

\begin{remark}
Note that property (1) is equivalent to the equality $\ad(D^{[p]})(E)=\ad(D)^p(E)$ for any local sections
$D, E \in  \Lambda_1$. In fact, by Jacobson's identity we have $\ad(D)^p = \ad(D^p)$, hence if
$D^{[p]} - D^p$ commutes with any $E \in \Lambda_1$, it commutes with any $E \in \Lambda$, since
$\Lambda$ is generated by $\Lambda_1$.
\end{remark}

\begin{remark} \label{centerLambda}
Let $F : X \to X$ denote the absolute Frobenius of $X$ and let $Z(\Lambda)$ denote the center of 
$\Lambda$. We note that the center $Z(\Lambda)$ does not have the structure of an $\SO_X$-module.
However, the left and right $\SO_X$-module structures on the direct image $F_*(Z(\Lambda))$ coincide, 
since for any local sections
$D \in \Lambda_1$ and $f \in \SO_X$ we have
$$[D, f^p] = \delta_D(f^p) = 0. $$
\end{remark}

\begin{proposition}
\label{prop:katzTangent}
For every local sections $D\in \Lambda_1$ and $f\in \SO_X$ we have
$$\delta_{fD}^{p-1}(f)=f\delta_D^{p-1}(f^{p-1}).$$
\end{proposition}

\begin{proof}
The relative tangent sheaf $T_{X/S} \cong \op{Der}_{\SO_S}(\SO_X, \SO_X)$ with the standard commutator 
is a $\SO_S$-Lie algebroid. Moreover this Lie algebroid is equipped with a
$p$-structure $\nu \mapsto \nu^p\in T_{X/S}$ (see also Remark \ref{derivations}). Thus, by the Hochschild identity (see \cite[Lemma 1]{Hoch}, 
\cite[Lemma 4.3]{Langer}, 
\cite[Lemma 2.1]{Scha16}),
we have for every local derivation $\nu\in T_{X/S}$ and every local section $f\in \SO_X$ the equality 
$$(f\nu)^{p}=f^p\nu^{p}+(f\nu)^{p-1}(f) \nu$$
in the associative $\SO_S$-algebra $\mathrm{End}_{\SO_S}(\SO_X)$.
On the other hand, we have the following identity from Deligne (cf. \cite[Proposition 5.3]{Katz70})
$$(f\nu)^p=f^p\nu^p+f\nu^{p-1}(f^{p-1})\nu.$$
Therefore we have that for every $\nu\in T_X$
$$(f\nu)^{p-1}(f)\nu = f\nu^{p-1}(f^{p-1})\nu.$$
If $\nu=0$, then clearly $(f\nu)^{p-1}(f) = f\nu^{p-1}(f^{p-1})=0$. Otherwise, the left-hand side and right-hand side of the 
equality are multiples of the same nonzero section of the torsion free sheaf $T_{X/S}$, so they are equal if and only if
$$(f\nu)^{p-1}(f)=f\nu^{p-1}(f^{p-1}).$$
Therefore, the latter equality holds for every local derivation $\nu\in T_{X/S}$ and every local section $f\in \SO_X$. 
The proposition is then obtained by applying the previous equality to $\nu=\delta_D$ and taking into account that
$f\delta_D=\delta_{fD}$, i.e. that the anchor map $\delta$ is $\SO_X$-linear.
\end{proof}

\begin{corollary}
\label{cor:DeligneLambda}
If $\Lambda$ is a restricted sheaf of differential operators on $X$ over $S$, then for every local 
sections $D\in \Lambda_1$ and  $f\in \SO_X$ we have
$$(fD)^{[p]}=f^pD^{[p]}+ f \delta_D^{p-1}(f^{p-1}) D.$$
\end{corollary}

\bigskip

\subsection{The map $\iota : \Lambda_1 \to Z(\Lambda)$} \label{defiotamap}

Using the $p$-structure on $\Lambda$, we can define the following map, generalizing the difference of $p$-th power maps on vector fields
\begin{eqnarray*}
\xymatrixrowsep{0.05pc}
\xymatrixcolsep{0.3pc}
\xymatrix{
\iota&:&\Lambda_1 \ar[rrr] &&&& \Lambda \\
&& D \ar@{|->}[rrr] &&&& \iota(D)=D^p-D^{[p]}. 
}
\end{eqnarray*}

\begin{proposition}
\label{prop:iota-p-linear}
The map $\iota:\Lambda_1\to \Lambda$ is a $p$-linear map, i.e., for every local sections $D,D_1,D_2\in \Lambda_1$ 
and $f\in \SO_X$ we have 
\begin{enumerate}[a)]
\item $\iota(D_1+D_2)=\iota(D_1)+\iota(D_2),$
\item $\iota(fD)=f^p\iota(D).$
\end{enumerate}
\end{proposition}

\begin{proof}
a) Let us apply Jacobson's identity in the associative ring $\Lambda(U)$, where $U$ is any open 
subset where $D_1$ and $D_2$ are both defined
$$(D_1+D_2)^p=D_1^p+D_2^p+\sum_{i=1}^{p-1}s_i(D_1,D_2).$$
On the other hand, as $[p]$ is a $p$-structure on $\Lambda$, we have
$$(D_1+D_2)^{[p]}=D_1^{[p]}+D_2^{[p]} + \sum_{i=1}^{p-1}s_i(D_1,D_2).$$
Therefore, subtracting one from the other yields
$$\iota(D_1+D_2)=(D_1+D_2)^p-(D_1+D_2)^{[p]} = D_1^p+D_2^p-D_1^{[p]}-D_2^{[p]}=\iota(D_1)+\iota(D_2).$$
b) Let us consider $f\in \SO_X=\Lambda_0$ as a local section of $\Lambda$. Then we can apply Deligne's identity (cf. \cite[Proposition 5.3]{Katz70}) in the associative ring $\Lambda(U)$ for an open subset $U$ such that $f\in \SO_X(U)$ and $D\in \Lambda(U)$ and we obtain
$$(fD)^p=f^pD^p+f \op{ad}(D)^{p-1}(f^{p-1}) D.$$
As the adjoint of $D$ applied to any local function is simply $\delta_D$, we obtain
$$(fD)^p=f^pD^p+f \delta_D^{p-1}(f^{p-1}) D.$$
On the other hand, by Corollary \ref{cor:DeligneLambda} we have
$$(fD)^{[p]}=f^pD^{[p]}+f \delta_D^{p-1}(f^{p-1}) D.$$
Therefore, subtracting one from the other yields
$$\iota(fD)=(fD)^p-(fD)^{[p]}=f^pD^p-f^p D^{[p]} = f^p\iota(D).$$
\end{proof}



\begin{proposition}
\label{cor:iotaCenter}
The image of $\iota$ lies in the center $Z(\Lambda)$ of $\Lambda$.
\end{proposition}

\begin{proof}
Using Jacobson's identity $\ad(D^p) = \ad(D)^p$ we obtain that for any local sections $D,E \in  \Lambda_1$ 
\begin{eqnarray*}
\ad(\iota(D)) (E) = \ad(D^p-D^{[p]}) (E) & = & \ad(D^p)(E) -\ad(D^{[p]})(E) \\
                                    &  = & \ad(D)^p (E)  -\ad(D)^p(E)  = 0.
\end{eqnarray*}
So $\iota(D)$ commutes with every element in $\Lambda_1$. As $\Lambda_1$ generates $\Lambda$, $\iota(D)$ commutes with 
every element in $\Lambda$.
\end{proof}

Observe that for each $f\in \SO_X$ we have $\iota(f)=f^{[p]}-f^p=0$ and that for each $f\in \SO_X$ and $D \in \Lambda_1$ 
we have
$$\iota(f+D)=\iota(f)+\iota(D)=\iota(D).$$
So $\iota$ factorizes through the quotient
$$\iota:\Lambda_1/\Lambda_0 = H \longrightarrow Z(\Lambda).$$
Then, as $\iota$ is a $p$-linear map, it induces an $\SO_X$-linear map
$$\iota: H \longrightarrow F_*(Z(\Lambda)),$$
where $F$ denotes the absolute Frobenius of $X$.
Moreover, $F_*(Z(\Lambda))$ is a commutative $\SO_X$-algebra (see Remark \ref{centerLambda}), so, 
by the universal property of the symmetric algebra, the map $\iota$ induces a map of sheaves of 
commutative $\SO_X$-algebras
$$\iota:\Sym^\bullet (H) \longrightarrow F_*(Z(\Lambda)).$$

\begin{proposition}
Suppose that $\Lambda$ is almost polynomial. Then the induced map $\iota:\Sym^\bullet(H) \to F_*(Z(\Lambda))$ 
is injective.
\end{proposition}

\begin{proof}
We note that the symbol map $\smb:\Lambda \longrightarrow \Gr^\bullet(\Lambda)\cong \Sym^\bullet(H)$ is 
a multiplicative (but not $\SO_X$-linear) map, so, composing with $\iota$, we obtain a multiplicative map
\begin{eqnarray*}
\xymatrix{
\Sym^\bullet(H) \ar[r]^-{\iota} \ar[dr]_{\smb(\iota)} & F_*(Z(\Lambda)) \ar[d]^{\smb}\\
& F_*(\Sym^\bullet(H)).
}
\end{eqnarray*}
To prove that $\ker(\iota)=0$ it is enough to prove that $\ker(\smb(\iota))=0$. As $\Lambda$ is 
almost polynomial, we have for every non-zero local $D\in H$ and any representative 
$\overline{D} \in \Lambda_1$ with $\smb(\overline{D})=D$ 
$$\smb(\overline{D}^p)=D^p\in \Sym^p(H).$$
So
$$\smb(\iota(D))=\smb(\overline{D}^p)=D^p\ne 0.$$
Moreover, for every local section $D\in \Sym^\bullet (H)$ there exist $D_1,\ldots,D_k \in H$ 
such that $D= D_1\cdots D_k+\tilde{D}$ with $\tilde{D}$ of degree $<k$. Therefore
$$\smb(\iota)(D)= \prod_{j=1}^k \smb(\iota)(D_j)= \prod_{j=1}^k D_j^p \ne 0.$$
\end{proof}

\section{Properties of almost abelian restricted sheaves of rings of differential operators}

Assume that the characteristic of $\mathbb{K}$ is $p>0$. Let $\pi : X \to S$ be a morphism between schemes
of finite type over $\mathbb{K}$. 

\subsection{Restricted $\SO_S$-Lie algebroid}

We need to recall some definitions (\cite{Hoch}, \cite[Section 3.1]{Ru}, \cite[Definition 4.2]{Langer}, 
\cite[Definition 2.2]{Scha16}).

\begin{definition} 
A restricted $\SO_S$-Lie algebroid on $X$  is a quadruple 
$(H,[-,-], \delta, [p])$
consisting of an $\SO_X$-module $H$, which is also a sheaf of
restricted $\SO_S$-Lie algebras, a map
$[p]: H \to H$ and an $\SO_X$-linear anchor map $\delta : H \to T_{X/S}$ satisfying the following
conditions for all local sections $f \in \SO_X$ and $D,D_1,D_2 \in H$
\begin{enumerate}
    \item $[D_1,fD_2] = f [D_1 , D_2] + \delta_{D_1}(f)D_2$,
    \item   $(fD)^{[p]}=f^pD^{[p]}+\delta_{fD}^{p-1}(f) D$.
\end{enumerate}
\end{definition}

\begin{remark} \label{derivations}
The standard example of restricted $\SO_S$-Lie algebroid on $X$ over $S$ is the relative tangent sheaf 
$T_{X/S} \cong \op{Der}_{\SO_S}(\SO_X, \SO_X)$ with the standard Lie bracket, $[p]$ the p-th power map and
$\delta$ the identity map. Note that condition (2) is then equivalent to the Hochschild identity (\cite[Lemma 1]{Hoch}).
\end{remark}

\subsection{Examples of almost abelian restricted sheaves of rings of differential operators}

We consider a restricted sheaf $\Lambda$ of rings of differential operators as in Definition \ref{defrestrictedlambda}.
In this subsection we assume that $\Lambda$ is almost abelian, i.e., the graded algebra $\op{Gr}^\bullet(\Lambda)$ is abelian.
Then for any two local sections $D_1,D_2\in \Lambda_1$ we have 
$$[\smb(D_1),\smb(D_2)]_{\op{Gr}^\bullet(\Lambda)}=0\in \Lambda_2/\Lambda_1,$$
so $[D_1,D_2]\in \Lambda_1$ and therefore $\Lambda_1$ with the induced commutator and anchor $\delta_D(f)=[D,f]$ for $D\in \Lambda_1$ and $f\in \SO_X$ becomes an $\SO_S$-Lie algebroid. In this case, conditions (1)-(3) of Definition \ref{defrestrictedlambda} are equivalent to 
asking that $(\Lambda_1,[-,-],\delta, [p])$ is a restricted $\SO_S$-Lie algebroid. Condition (4) is then equivalent to asking that 
the inclusion of $\SO_S$-Lie algebroids
$$(\SO_X, [-,-]=0, \delta=0, (-)^p ) \hookrightarrow (\Lambda_1,[-,-], \delta , [p])$$
is a homomorphism  of restricted $\SO_S$-Lie algebroids.

We first need some information on the universal Lie polynomials used in Definition \ref{defrestrictedlambda}.

\begin{lemma}
\label{lemma:extensionOrder0}
Let $\Lambda$ be any sheaf of rings of differential operators on $X$ over $S$. 
Let $D\in \Lambda_1$ and $f\in \SO_X$. Then for every $i<p-1$
$$s_i(D,f)=0$$
and
$$s_{p-1}(D,f)=\delta_D^{p-1}(f).$$
\end{lemma}

\begin{proof}
In any associative algebra of characteristic $p$ it is a classical result that we can write 
the Lie polynomial $s_i(x_1,x_2)$ for $1 \leq i \leq p-1$ as follows
$$s_i(x_1,x_2)=-\frac{1}{i} \sum_{\begin{array}{c}
\sigma:\{1,\ldots,p-1\} \to \{1,2\}\\
|\sigma^{-1}(1)|=i
\end{array}} \ad(x_{\sigma(1)}) \cdots \ad(x_{\sigma(p-1)}) (x_2).$$
Observe that for $x_1=D\in \Lambda_1$ and $x_2=f\in \SO_X$ we have the following equalities
$$\ad(x_1)(x_2)=\delta_D(f)\in \SO_X, \ \
\ad(x_1)(x_1)=0, \ \
\ad(x_2)(g)=0 \ \  \forall g \in \SO_X.$$
In particular, observe that for any indices $i$ and $j$
$$\ad(x_i)(x_j)\in \SO_X, \ \  \ad(x_i)(g)\in \SO_X \ \ \forall g\in \SO_X.$$
Thus, for  $i=1,2$ and $g\in \SO_X$
$$\ad(x_2)\ad(x_i)(g)=0.$$
In particular, if $\sigma(j)=2$ for some $j<p-1$ we have that
$$\ad(x_{\sigma(j+1)})\cdots \ad(x_{\sigma(p-1)})(x_2)\in \SO_X.$$
So
$$\ad(x_2)\ad(x_{\sigma(j+1)})\cdots \ad(x_{\sigma(p-1)})(x_2)=0,$$
and the corresponding summand in the expression of $s_i(D,f)$ would be zero. Similarly, if $\sigma(p-1)=2$ we have $\ad(x_2)(x_2)=0$ and the whole expression is zero. Thus for the sum to be non-zero 
we must have $\sigma(j)=1$ for all $j=1,\ldots,p-1$. Finally, we have that for $i=p-1$
$$s_{p-1}(D,f)=-\frac{1}{p-1} \ad(D) \cdot \ad(D) (f) = -\frac{1}{p-1} \delta_D^{p-1}(f) = \delta_D^{p-1}(f).$$
\end{proof}

\begin{proposition} \label{HrestrictedLiealgebroid}
If $\Lambda$ is an almost abelian restricted ring of differential operators on $X$ over $S$, then 
$H = \Lambda_1/\Lambda_0$ inherits a restricted $\SO_S$-Lie algebroid structure 
$(H,[-,-]_H, \delta, [p])$ such that the short exact sequence  (\ref{sbmap}) becomes an exact sequence
of restricted $\SO_S$-Lie algebroids.
\end{proposition}

\begin{proof}
First of all, for each $D_1,D_2\in H$ define $[D_1,D_2]_H = \smb([\overline{D_1},\overline{D_2}]_{\Lambda})$ 
for any $\overline{D_i}$ such that $\smb(\overline{D_i})=D_i$ for $i=1,2$. In order to prove that it 
is well-defined observe that for each $f_1,f_2\in \SO_X$ we have
$$\smb([f_1+\overline{D_1},f_2+\overline{D_2}]_{\Lambda})=\smb([\overline{D_1},\overline{D_2}]_{\Lambda_1}+\delta_{\overline{D_1}}(f_2)-\delta_{\overline{D_2}}(f_1)) = \smb([\overline{D_1},\overline{D_2}]_{\Lambda_1}).$$
Similarly, as $\delta_f(g)=[f,g]_{\Lambda}=0$ for each $f,g \in \SO_X$, clearly $\delta$ factorizes through $H$.

Finally, define $D^{[p]}=\smb(\overline{D}^{[p]})$. Then for each $f\in \SO_X$ we have that, using property (2) of 
the definition of $p$-structure and Lemma \ref{lemma:extensionOrder0} we have
\begin{multline*}
\smb((\overline{D}+f)^{[p]}) = \smb(\overline{D}^{[p]}+f^p+\sum_{i=1}^{p-1}s_i(D,f))=\smb(\overline{D}^{[p]}+f^p + \delta_{\overline{D}}^{p-1}(f))=\smb(\overline{D}^{[p]}).
\end{multline*}
By construction, taking the symbol of the corresponding expressions in (1), (2) and (3), those properties are also satisfied for the induced $p$-structure on $H$, and the symbol map $\smb:\Lambda_1\longrightarrow  H$ is a morphism of restricted $\SO_S$-Lie algebroids.
\end{proof}

\bigskip

On the other hand, let us consider a restricted $\SO_S$-Lie algebroid
$(H,[-,-], \delta, [p])$. Then the universal enveloping algebra\footnote{This sheaf of algebras is called the universal enveloping algebra of differential operators associated to $H$ in \cite{Langer}}
$\Lambda_H$ of the $\SO_S$-Lie algebroid $H$, as defined e.g. in \cite[Section 4.3]{To} or 
\cite[page 515]{Langer}, becomes a split almost polynomial  
restricted sheaf of rings of differential operators on $X$ over $S$ by taking the $p$-structure as follows: we have a splitting as $\SO_X$-modules
$$(\Lambda_H)_1=\SO_X \oplus H,$$ 
and we define for every $D\in H$ and every $f\in \SO_X$
$$(f+D)^{[p]}=f^p+D^{[p]}+\delta_D^{p-1}(f).$$
We will show in the next proposition that this map endows $\Lambda_H$ with the structure of a restricted sheaf of rings of 
differential operators. First we will need two lemmas.

\begin{lemma}
\label{lemma:inducedAditivity}
For any local sections $f_1,f_2\in \SO_X$ and $D_1,D_2\in H$ we have the following equality in $\Lambda_H$
$$\delta_{D_1}^{p-1}(f_1)+\delta_{D_2}^{p-1}(f_2)+\sum_{i=1}^{p-1} s_i(f_1+D_1,f_2+D_2)=\sum_{i=1}^{p-1} s_i(D_1,D_2)+\delta_{D_1+D_2}^{p-1}(f_1+f_2).$$
\end{lemma}

\begin{proof}
We will use Jacobson's formula to compute $(f_1+D_1+f_2+D_2)^p\in \Lambda_H$ in two different ways. 
On one hand, taking into account Lemma \ref{lemma:extensionOrder0} we have
\begin{multline*}
((f_1+D_1)+(f_2+D_2))^p = (f_1+D_1)^p+(f_2+D_2)^p+\sum_{i=1}^{p-1}s_i(f_1+D_1,f_2+D_2)\\
=f_1^p+D_1^p+\delta_{D_1}^{p-1}(f_1) + f_2^p+D_2^p+\delta_{D_2}^{p-1}(f_2)+\sum_{i=1}^{p-1}s_i(f_1+D_1,f_2+D_2).
\end{multline*}
On the other hand, we have
\begin{multline*}
((f_1+f_2)+(D_1+D_2))^p = f_1^p+f_2^p + (D_1+D_2)^p + \delta_{D_1+D_2}^{p-1}(f_1+f_2)\\
=f_1^p+f_2^p+D_1^p+D_2^p+\sum_{i=1}^{p-1} s_i(D_1,D_2)+\delta_{D_1+D_2}^{p-1}(f_1 + f_2).
\end{multline*}
Subtracting both expressions yields the desired equality.
\end{proof}

\begin{lemma}
\label{lemma:iteratedDeltaDistributive}
For any local sections $f,g\in \SO_X$ and any local section $D\in H$ we have
$$\delta_{gD}^{p-1}(gf)=g^p\delta_D^{p-1}(f)+\delta_{gD}^{p-1}(g)f.$$
\end{lemma}

\begin{proof}
As it is an equality of local sections in $\SO_X$, it is enough to prove that the difference of the sections is zero on an open set. In particular, as the equality clearly holds if $f=0$, we can assume that $f \not= 0$
and restrict to the open subset where $f$ is invertible. 
Then $D'=D/f$ is an element of $H$ and we have the following two identities as a consequence of the $p$-structure on $H$
$$((gf)D')^{[p]} = g^pf^p(D')^{[p]}+\delta_{gfD'}^{p-1}(gf) D',$$
$$(g(fD'))^{[p]}=g^p (fD')^{[p]}+\delta_{gfD'}^{p-1}(g)(fD')=g^pf^p(D')^{[p]}+g^p\delta_{fD'}^{p-1}(f)D' + \delta_{gfD'}^{p-1}(g)fD'.$$
Subtracting and considering coefficients of $D'$ yields the equality
$$\delta_{gfD'}^{p-1}(gf) = g^p\delta_{fD'}^{p-1}(f)+\delta_{gfD'}^{p-1}(g)f.$$
Taking into account that $D=fD'$ we obtain the result.
\end{proof}

\begin{proposition} \label{resLambda}
Let $H$ be a restricted $\SO_S$-Lie algebroid on $X$ over $S$. Then 
the map $[p]:\SO_X\oplus H \longrightarrow \SO_X \oplus H$ defined by 
$$(f+D)^{[p]}=f^p+D^{[p]}+\delta_D^{p-1}(f) $$
is a $p$-structure for the universal enveloping algebra $\Lambda_H$ making the symbol map $\smb:(\Lambda_H)_1 \longrightarrow H$ a morphism 
of restricted $\SO_S$-Lie algebroids.
\end{proposition}

\begin{proof}
It will be enough to check the four properties of Definition \ref{defrestrictedlambda}.
\begin{enumerate}
    \item
By Jacobson's formula in $\Lambda_H$ and by  Lemma \ref{lemma:extensionOrder0} we have the equality
$$(f+D)^p=f^p+D^p+\delta_D^{p-1}(f).$$
So
$$\ad(f+D)^p=\ad(f^p)+\ad(D^{[p]})+\ad(\delta_D^{p-1}(f)) = \ad((f+D)^{[p]}).$$
\item 
To prove additivity we use Lemma \ref{lemma:inducedAditivity} to obtain
\begin{eqnarray*}
& & (f_1+D_1+f_2+D_2)^{[p]} \\ 
& = &((f_1+f_2)+(D_1+D_2))^{[p]} = f_1^p+f_2^p+(D_1+D_2)^{[p]} + \delta_{D_1+D_2}^{p-1}(f_1+f_2) \\
& = & f_1^p+f_2^p+D_1^{[p]}+D_2^{[p]} +\sum_{i=1}^{p-1} s_i(D_1,D_2)+\delta_{D_1+D_2}^{p-1}(f_1 + f_2) \\
& = & f_1^p+f_2^p+D_1^{[p]}+D_2^{[p]} +  \delta_{D_1}^{p-1}(f_1)+\delta_{D_2}^{p-1}(f_2)+\sum_{i=1}^{p-1} s_i(f_1+D_1,f_2+D_2)\\
& =  & (f_1+D_1)^{[p]}+ (f_2+D_2)^{[p]} + \sum_{i=1}^{p-1}s_i(f_1+D_1,f_2+D_2).
\end{eqnarray*}
\item 
Let $f,g\in \SO_X$ and $D\in H$. Then by Lemma \ref{lemma:iteratedDeltaDistributive} we have
\begin{eqnarray*}
 & & (g(f+D))^{[p]} \\
& = & (gf+gD)^{[p]} = g^pf^p+(gD)^{[p]}+\delta_{gD}^{p-1}(gf)\\
& = & g^pf^p+g^pD^{[p]} + \delta_{gD}^{p-1}(g)D+\delta_{gD} ^{p-1}(gf)\\
& = & g^pf^p+g^pD^{[p]} + \delta_{gD}^{p-1}(g)D+g^p\delta_D^{p-1}(f)+\delta_{gD}^{p-1}(g)f\\
& = & g^p(f+D)^{[p]} + \delta_{gD}^{p-1}(g)(f+D) = g^p(f+D)^{[p]}+\delta_{g(f+D)}^{p-1}(g)(f+D).
\end{eqnarray*}
\item This property is obvious by taking $D= 0$.
\end{enumerate}
\end{proof}

To summarize, we have shown that the definition of a $p$-structure on the universal enveloping algebra $\Lambda_H$ of a restricted 
$\SO_S$-Lie algebroid $H$, as well as the usual notion of $p$-th power for crystalline differential operators are particular cases of 
our general definition of a $p$-structure for a restricted sheaf of rings of differential operators (Definition \ref{defrestrictedlambda}).

\section{Some examples of restricted sheaves of rings of differential operators}

In this section we assume that $\pi : X \to S$ is a smooth morphism.

\subsection{Sheaf of crystalline differential operators $\SD_{X/S}$} \label{sheafofcrystallinediffop}

The sheaf of crystalline differential operators  (see e.g. \cite{BMR}) 
$$  \Lambda^{dR} = \SD_{X/S} $$
is a split almost polynomial restricted sheaf of rings of differential operators. Its associated restricted 
$\SO_S$-Lie algebroid $\Lambda^{dR}_1 / \Lambda^{dR}_0$ is the relative tangent sheaf
$T_{X/S}$, taking the commutator as the Lie bracket of vector fields and taking the identity as the anchor map. 
The $\SD_{X/S}$-modules correspond to coherent $\SO_X$-modules with a relative integrable connection.

For every derivation $\nu\in T_{X/S}$ the $p$-th power $\nu^p$ is again a derivation, since by applying Leibniz rule, we have for every 
local section $f,g\in \SO_X$
$$\nu^p(fg)=\sum_{k=0}^p \binom{p}{k} \nu^k(f) \nu^{p-k}(g) = \nu^p(f) g + f \nu^p(g)$$
so taking $\nu^{[p]}=\nu^p$ gives us a $p$-structure  $[p]:T_{X/S} \to T_{X/S}$ endowing $T_{X/S}$ with the structure of a 
restricted $\SO_S$-Lie algebroid $(T_{X/S},[-,-], \mathrm{id}_{T_{X/S}},[p])$ and, therefore, inducing a $p$-structure on $\SD_X$.

\subsection{Trivial $p$-structure on the symmetric algebra} \label{trivialpstructure}

Given a locally free $\SO_X$-module $H$ over $X$, the symmetric algebra
$$ \Lambda^{\op{Higgs}} = \Sym^\bullet(H)$$
is a split almost polynomial restricted sheaf of rings of differential operators, when taking the trivial $p$-stucture on 
$\Lambda_1 = \SO_X \oplus H$, i.e. we take $[p]:H\to H$ to be the zero map on H
$$D^{[p]}=0.$$
Then a $\Lambda^{\op{Higgs}}$-module corresponds to a $H^\vee$-valued Higgs bundle $(E,\phi)$, where $E$ 
is a vector bundle over $X$ and $\phi:E\to E\otimes H^\vee$ is a morphism of $\SO_X$-modules satisfying 
$\phi \wedge \phi=0$.

As $\Lambda^{\op{Higgs}}$ is abelian, we have
$$\ad_{\Lambda_1}(D)^p =0 = \ad_{\Lambda_1}(D^{[p]}).$$
Moreover $s_i(D_1,D_2)=0$ for all $D_1,D_2\in H$, so
$$(D_1+D_2)^{[p]}=0=D_1^{[p]}+D_2^{[p]} =D_1^{[p]}+D_2^{[p]} + \sum_{i=1}^{p-1} s_i(D_1,D_2).$$
Finally, $\Lambda^{\op{Higgs}}$ being abelian implies $\delta=0$, so we trivially have
$$0 = (fD)^{[p]} =f^pD^{[p]}+\delta_{fD}^p(f)D = 0.$$

\subsection{$p$-structure on the reduction to the associated graded of $\SD_{X/S}$}

By the classical Rees construction applied to the filtered sheaf $\Lambda^{dR} = \SD_{X/S}$ (see Subsection 4.1) we obtain a sheaf of rings 
over $X \times \mathrm{Spec}(\mathbb{K}[t]) = X \times \mathbb{A}^1$ defined as
$$ \Lambda^{dR,R} = \bigoplus_{i \geq 0} t^i \Lambda_i,$$
where $t$ acts by multiplication with $t$ on  $\Lambda^{dR,R}$ using the inclusions $\Lambda_i \subset \Lambda_{i+1}$.
Then by construction the fibers over the closed points $0$ and $1$ of $\mathbb{A}^1$ equal 
$$ (\Lambda^{dR,R})_0 = \Sym^\bullet (T_{X/S}) \quad   \text{and}  \quad   (\Lambda^{dR,R})_1 = \mathcal{D}_{X/S} = \Lambda^{dR}.$$
We observe that $\Lambda^{\op{dR},R}$ is a split almost polynomial sheaf of rings of differential operators 
on $X \times \mathbb{A}^1$ relative to $S \times \mathbb{A}^1$ such that the fiber over each $\lambda \in \mathbb{A}^1$ 
corresponds to the universal enveloping algebra of the $\SO_S$-Lie algebroid $(T_{X/S}, \lambda [-,-],\lambda \mathrm{id}_{T_{X/S}})$. 

\bigskip
We can endow $\Lambda^{dR,R}$ with a $p$-structure as follows. We note that
$$\Lambda^{dR,R}_1 = \SO_{X \times \mathbb{A}^1} \oplus T_{X \times \mathbb{A}^1 / S \times \mathbb{A}^1} \quad
\text{and} \quad  T_{X \times \mathbb{A}^1 / S \times \mathbb{A}^1}  =  
T_{X/S}.$$
Then the $p$-structure on $\Lambda^{dR,R}_1$ over $X \times \mathbb{A}^1$ is defined by
$$ [p]^R: T_{X/S} \to  T_{X/S} \quad \quad D^{[p]^R} = t^{p-1} D^{[p]},  $$
where $t$ is the coordinate on $\mathbb{A}^1$ and $D^{[p]}$ is the $p$-th power of the relative vector field
$D \in T_{X/S}$. 
By construction of $\Lambda^{\op{dR},R}$ the commutator of elements in $\Lambda^{\op{dR},R}_1$ is the commutator of differential operators multiplied by the coordinate $t$, i.e., for every $D\in \Lambda^{\op{dR},R}_1$
$$\ad_{\Lambda^{\op{dR},R}_1}(D)= t \ad_{\Lambda^{dR}_1}(D)$$
Moreover, as the Lie polynomials $s_i(x,y)$ are homogeneous of degree $p-1$, we have
$$s_i^{\Lambda^{\op{dR},R}}(x,y)= t^{p-1}s_i^{\Lambda^{dR}}(x,y).$$
Therefore, the following equalities hold for any local sections $D \in T_{X/S}$ and 
$f \in \SO_{X \times \mathbb{A}^1}$

$$\ad_{\Lambda_1^{\op{dR,R}}}(D^{[p]^R})= t \ad_{\Lambda^{dR}_1}(t D^{[p]})=
t^p \ad_{\Lambda^{dR}_1}(D)^p  = (t \ad_{\Lambda^{dR}_1}(D))^p = \ad_{\Lambda_1^{\op{dR},R}}(D)^p,$$
\begin{multline*}
(D_1+D_2)^{[p]^R}= t^{p-1}(D_1+D_2)^{[p]} = t^{p-1}D_1^{[p]}+ t^{p-1}D_1^{[p]}+
\sum_{i=1}^{p-1} t^{p-1}s_i^{\Lambda^{dR}}(D_1,D_2)\\
=D_1^{[p]^R}+D_2^{[p]^R}+
\sum_{i=1}^{p-1}s_i^{\Lambda^{\op{dR},R}}(D_1,D_2),
\end{multline*}

\begin{multline*}
(fD)^{[p]^R}= t^{p-1}(fD)^{[p]}=t^{p-1}f^pD^{[p]} + t^{p-1}\left(\delta_{fD}^{\Lambda^{dR}} \right)^{p-1}(f)D \\ 
= f^pD^{[p]^R}+\left(\delta_{fD}^{\Lambda^{\op{dR},R}} \right)^{p-1}(f)D.
\end{multline*}

This proves that $[p]^R$ is a $p$-structure for $\Lambda^{\op{dR},R}$. 

\bigskip

\subsection{$p$-structure on the reduction to the associated graded: general case} \label{redgrad}

More generally, let $\Lambda = \Lambda_H$ be the restricted sheaf of rings of differential operators over $X$ given as the 
universal enveloping algebra of a restricted $\SO_S$-Lie algebroid $(H,[-,-],\delta, [p])$ --- see Proposition \ref{resLambda}. 
Consider the Rees construction $\Lambda^R$ over $X\times \mathbb{A}^1$ relative to 
$S \times \mathbb{A}^1$ of the filtered sheaf of rings $\Lambda$. Then the fiber of $\Lambda^R$ over 
$\lambda \in \mathbb{A}^1$ 
is the universal enveloping algebra of the $\SO_S$-Lie algebroid $(H,\lambda[-,-], 
 \lambda \delta)$. We also note that $\Lambda_1^R / \Lambda_0^R = p_X^*(H)$, where $p_X : X \times \mathbb{A}^1 \rightarrow X$
is the projection onto $X$. The anchor map $\delta^R$ of $\Lambda^R$ equals
$$ \delta^R = t \delta : \Lambda_1^R / \Lambda_0^R = p_X^*(H) \rightarrow p_X^*(T_{X/S}).$$
Then the previous argument proves that the map $[p]^R : p_X^*(H)  \to p_X^*(H)$  over $X \times \mathbb{A}^1$ 
 given by
$$D^{[p]^R}= t^{p-1} D^{[p]}$$
is a $p$-structure for $\Lambda^R$. This also yields an explicit deformation 
of the $p$-structure on $\Lambda$ to the trivial 
$p$-structure on $\Gr^\bullet(\Lambda) \cong \Sym^\bullet(H)$.

\bigskip

\subsection{$p$-structure on the Atiyah algebroid of a line bundle} \label{twisteddiffop}

Let us study an example which is almost polynomial, but not split.
Let $L$ be a line bundle on $X$ and take $\Lambda$ to be the sheaf of crystalline differential operators on $L$, i.e., 
the subalgebra 
$$\Lambda = \mathcal{D}_{X/S}(L) \subset \End_{\SO_S}(L)$$
generated by the relative Atiyah algebroid $\op{At}_{X/S}(L)=\op{Diff}_{\SO_S}^1(L,L)$. Note that
$$\Lambda_1 = \op{At}_{X/S}(L).$$
Local sections of $\op{At}_{X/S}(L)$ can be identified with local sections $D \in \End_{\pi^{-1}(\SO_S)}(L)$ 
such that for each $f \in \SO_X$, $[D,f]\in \End_{\SO_X}(L)= \SO_X = \SD^0(L)$.
Then, for every $D\in \op{At}_{X/S}(L)$ let us denote by $\delta_D:\SO_X\to \SO_X$ the map
$$\delta_D(f)=[D,f]\in \SO_X.$$
Observe that, as $\Lambda$ is associative, we have that for each $f,g\in \SO_X$
\begin{eqnarray*}
\delta_D(fg)=[D,fg] & = & Dfg-fgD = Dfg-fDg+fDg-fgD \\
                    & = & [D,f]g+f[D,g] = \delta_D(f)g+f\delta_D(g)
\end{eqnarray*}
thus, $\delta_D$ is a $\SO_S$-derivation and we can consider the map  $\delta: \op{At}_{X/S}(L) \longrightarrow T_{X/S}$. 
So we obtain the short exact sequence
\begin{equation} \label{atiyahalgebra}
0\longrightarrow \SO_X \longrightarrow \op{At}_{X/S}(L)  \stackrel{\delta}{\longrightarrow} T_{X/S} \longrightarrow 0.
\end{equation}
Thus the triple $(\op{At}_{X/S}(L), [-,-], \delta)$ becomes a $\SO_S$-Lie algebroid. We will now endow this Lie algebroid with a 
$p$-structure.

\begin{lemma}
Let $D\in \op{At}_{X/S}(L)$. Then for every $f\in \SO_X$, $[D^p,f]\in \SO_X$, so $D^p$ can be identified with an element in $\op{At}_{X/S}(L)$ 
that we will denote as $D^{[p]}$.
\end{lemma}

\begin{proof}
As $\Lambda$ is an associative $\SO_X$-algebra of characteristic $p$ we can apply Jacobson's formula and we have that for every $D\in \op{At}_{X/S}(L)$ 
and every $f\in \SO_X$
$$[D^p,f]=\ad(D^p)(f) = \ad(D)^p(f) = \delta_D^p(f) \in \SO_X.$$
\end{proof}

\begin{proposition}
The map $[p]: \op{At}_{X/S}(L) \to \op{At}_{X/S}(L)$ described in the previous lemma is a $p$-structure for $\Lambda$.
\end{proposition}

\begin{proof}
Property (1) was proved in the previous lemma. 
For the additivity property (2), observe that in $\Lambda$ Jacobson's formula yields
$$(D_1+D_2)^p = D_1^p+D_2^p+\sum_{i=1}^{p-1}s_i(D_1,D_2).$$
As this is indeed an equality in the $\SO_X$-algebra $\Lambda$, the commutator of the left and right side of the equation with an element of $\SO_X$ must yield the same element of $\SO_X$, so both left and right sides remain equal under the identification of $D_i^p$ with the 
corresponding element $D_i^{[p]}\in \op{At}_{X/S}(L)$.
For (3), since $\Lambda$ is associative, we can apply Deligne's identity \cite[Proposition 5.3]{Katz70} 
and we obtain that
$$(fD)^p=f^pD^p+f \ad(D)^{p-1}(f^{p-1}) D = f^pD^p+ f \delta_D^{p-1}(f^{p-1}) D.$$
Now, applying Proposition \ref{prop:katzTangent} we have that
$$f^pD^p+ f \delta_D^{p-1}(f^{p-1}) D=f^pD^p+\delta_{fD}^{p-1}(f)D$$
and, applying a similar argument to the previous property, we obtain the desired equality.
Finally, it is trivial by construction that for every $f\in \SO_X$, $f^{[p]}=f^p$.
\end{proof}

Finally, we mention that $\Lambda = \mathcal{D}_{X/S}(L)$ coincides with the Sridharan enveloping 
algebra $\Lambda_{\op{At}_{X/S}(L)}$ associated to the non-split extension \eqref{atiyahalgebra} of the Lie algebroid $T_{X/S}$ by $\SO_X$ as constructed in \cite[Section 4.3]{To} (see also \cite[Example 3.2.3]{ToPhD} for this particular case) or \cite[page 516]{Langer}.

\bigskip

\subsection{$p$-structures on the symmetric algebra} \label{pstrsymmalg}

Returning to the abelian setting, let us fix $\Lambda=\Sym^\bullet (H)$ for some locally free $\SO_X$-module 
$H$ and let us study the possible $p$-structures on $\Lambda$. As before, $\Lambda$ being abelian 
implies that for any $D \in H$
$$\ad_{\Lambda_1}(D)^p =0 = \ad_{\Lambda_1}(D^{[p]})$$
and for any $D_1,D_2\in H$, $s_i(D_1,D_2)=0$. Moreover, for any $D\in H$, $\delta_D=0$. 
Therefore, the conditions for a map $[p]:H\longrightarrow H$ to endow $\Lambda$ 
with a $p$-structure are the following
\begin{enumerate}
\item $(D_1+D_2)^{[p]}=D_1^{[p]}+D_2^{[p]}$,
\item $(fD)^{[p]}=f^pD^{[p]}$.
\end{enumerate}
So a $p$-structure on $\Sym^\bullet(H)$ is given by a $p$-linear map from $H$ to $H$,
or equivalently by an $\SO_X$-linear map 
$$ \alpha : F^*H \to H,$$
where $F$ denotes the absolute Frobenius of $X$.



\subsection{Classification of $p$-structures on a general $\Lambda$}

In this subsection we will describe all $p$-structures on a given sheaf of rings of differential operators $\Lambda$.

\begin{proposition}
\label{prop:classpStr}
Let $[p]: \Lambda_1 \to \Lambda_1$ be a $p$-structure for $\Lambda$. Then any other $p$-structure $[p]':\Lambda_1 \to \Lambda_1$ is given by
$$[p]'=[p]+\varphi\circ \smb$$
where $\varphi:H \longrightarrow Z(\Lambda_1)$ is a $p$-linear map from $H=\Lambda_1/\Lambda_0$ to the centralizer 
$Z(\Lambda_1)$ of $\Lambda_1$ in $\Lambda_1$.
\end{proposition}

\begin{proof}
We put $\varphi(D)=D^{[p]}-D^{[p]'}$. Then for every local sections $D,E\in \Lambda_1$ we have
$$\ad(\varphi(D))(E)=\ad(D^{[p]})(E)-\ad(D^{[p]'})(E)=\ad(D)^p(E)-\ad(D)^p(E)=0.$$
So $\varphi(D)\in Z(\Lambda_1)$ for every $D\in \Lambda_1$. Let $D_1,D_2\in \Lambda_1$. Then
$$\varphi(D_1+D_2)=D_1^{[p]}+D_2^{[p]}+\sum_{i=0}^{p-1}s_i(D_1,D_2)-D_1^{[p]'}-D_2^{[p]'}-\sum_{i=0}^{p-1}s_i(D_1,D_2)=\varphi(D_1)+\varphi(D_2).$$
Similarly
$$\varphi(fD)=f^pD^{[p]} + \delta_{fD}^{p-1}(f) D - f^pD^{[p]'}- \delta_{fD}^{p-1}(f)D = f^p\varphi(D).$$
So $\varphi:\Lambda_1\to \Lambda_1$ is $p$-linear. Moreover, clearly
$$\varphi(f)=f^{[p]}-f^{[p]'}=f^p-f^p=0.$$
So $\varphi$ factors through the quotient $\varphi:H \longrightarrow Z(\Lambda_1)$.

Conversely, let $[p]:\Lambda_1\to \Lambda_1$ be a $p$-structure on $\Lambda$ 
and let $\varphi:H\to Z(\Lambda_1)$ be a $p$-linear map. We then define $D^{[p]'}=D^{[p]}+\varphi(\smb(D))$. 
Then for every local section $D_1,D_2,D,E\in \Lambda_1$ and every local section $f\in \SO_X$
$$\ad(D^{[p]'})(E)=\ad(D^{[p]})+\ad(\varphi(\smb(D)))(E) = \ad(D)^p(E),$$
\begin{multline*}
(D_1+D_2)^{[p]'}=(D_1+D_2)^{[p]}+\varphi(\smb(D_1)+\smb(D_2))\\
=D_1^{[p]}+D_2^{[p]}+\sum_{i=0}^{p-1}s_i(D_1,D_2)+\varphi(\smb(D_1))+\varphi(\smb(D_2))
=D_1^{[p]'}+D_2^{[p]'}+\sum_{i=0}^{p-1}s_i(D_1,D_2)
\end{multline*}
\begin{multline*}
(fD)^{[p]'}=(fD)^{[p]}+\varphi(\smb(fD))=f^pD^{[p]}+\delta_{fD}^{p-1}(f)D + f^p\varphi(\smb(D))  \\ 
= f^pD^{[p]'}+\delta_{fD}^{p-1}(f)D,
\end{multline*}
$$f^{[p]'}=f^{[p]}+\varphi(\smb(f))=f^p+\varphi(0)=f^p.$$
So $[p]':\Lambda_1\to \Lambda_1$ induces a  $p$-structure on $\Lambda$.
\end{proof}

\begin{corollary}
The $p$-structures on $\Lambda^{dR} = \SD_{X/S}$ are classified by global $1$-forms 
$\omega\in H^0(F^*\Omega^1_{X/S})$ and are given by
$$(f+v)^{[p]'} = f^p+v^{[p]} + v^{p-1}(f) + \omega(F^*v)$$
for $f \in \SO_X$ and $v \in T_{X/S}$,
where $[p]:T_{X/S} \to T_{X/S}$ denotes the canonical $p$-structure on the relative tangent bundle given by the 
$p$-th power of vector fields.
\end{corollary}

\begin{proof}
We know that the $p$-th power on $T_{X/S}$ induces a $p$-structure $[p]$ on $\SD_{X/S}$ given by
$$(f+v)^{[p]} = f^p+v^{[p]} + v^{p-1}(f)$$
for $f \in \SO_X$ and $v \in T_{X/S}$. So by Proposition \ref{prop:classpStr} any other $p$-structure is given by adding 
a $\SO_X$-linear map $\varphi:F^*T_X \longrightarrow Z(\SD_{X/S}^1)$ composed with the symbol. Let us compute the center $Z(\SD_{X/S}^1)$. 
Any element of $Z(\SD_{X/S}^1)$ has to commute in particular with all elements in $\SD_{X/S}^0=\SO_X$. But the elements of $T_{X/S}$ that commute with $\SO_X$ are 
those in the kernel of the anchor  map $\delta: T_{X/S} \to T_{X/S}$, which is the identity map. Thus we obtain that $Z(\SD_{X/S}^1)\subset \SO_X$ and we have
$$F_*(Z(\SD_{X/S}^1)) = \SO_X = \ker(d:F_*\SO_X \to F_*\Omega_{X/S}^1).$$
Therefore, any other $p$-structure $[p]'$ must equal $[p] + \varphi \circ \smb$, where $\varphi:F^*T_{X/S} \longrightarrow F^*\SO_X$
is $\SO_X$-linear, which corresponds a global $1$-form in $H^0(F^*\Omega^1_{X/S})$, yielding the result.
\end{proof}


\bigskip

\section{$p$-curvature of a restricted $\Lambda$-module}

Let $\Lambda$ be a sheaf of rings of differential operators on $X$ over $S$ and let $E$ be a coherent
$\SO_X$-module.

\begin{definition}
A $\Lambda$-module structure on $E$ is an $\SO_X$-linear map
$$ \nabla : \Lambda \otimes_{\SO_X} E \longrightarrow E$$
satisfying the usual module axioms and such that the $\SO_X$-module structure on $E$ induced by $\SO_X \to
\Lambda$ coincides with the original one. 
\end{definition}

We will denote a $\Lambda$-module $E$ by $(E, \nabla)$ and for any local section $D \in \Lambda$ the 
$\SO_S$-linear endomorphism of $E$ induced by the action of $D$ will be denoted by  $\nabla_D \in \End_{\SO_S}(E)$. Given a $\Lambda$-module $(E, \nabla)$
and a local section $D \in \Lambda_1$ we define 
the $p$-curvature $\psi_\nabla(D):E\longrightarrow E$ as the map
$$\psi_\nabla(D)= (\nabla_D)^p- \nabla_{D^{[p]}} \in \End_{\SO_S}(E).$$
We observe that we can define the $p$-curvature in terms of the map $\iota:\Lambda_1\to \Lambda$ 
defined in Subsection \ref{defiotamap} as follows
$$\psi_\nabla(D) = (\nabla_D)^p-\nabla_{D^{[p]}}=\nabla_{D^p}-\nabla_{D^{[p]}}=\nabla_{D^p-D^{[p]}}=\nabla_{\iota(D)}.$$

\begin{proposition}
For any $D\in \Lambda_1$, $\psi_\nabla(D):E\to E$ is an $\SO_X$-linear map.
\end{proposition}

\begin{proof}
By definition the $\SO_X$-module structure induced by the action of $\Lambda$ on $E$ coincides 
with the $\SO_X$-module structure of $E$, so for any local sections $s\in E$ and $f\in \SO_X$ we have
$$fs=\nabla_f(s).$$
Moreover, as $\iota(D)\in Z(\Lambda)$ we have for any local section $D\in \Lambda_1$ 
\begin{eqnarray*}
\psi_\nabla(D)(fs) & = & \nabla_{\iota(D)} \circ \nabla_f(s) = \nabla_{\iota(D) f}(s) \\  
                   & = & \nabla_{f\iota(D)}(s) = \nabla_f \circ \nabla_{\iota(D)}(s)=f\psi_\nabla(D)(s).
\end{eqnarray*}
\end{proof}

This, together with Proposition \ref{prop:iota-p-linear} and the fact that $\iota$ factors through the symbol, 
proves that the $p$-curvature induces a $p$-linear map
$$\psi_\nabla: H\longrightarrow \End_{\SO_X}(E),$$
where $H = \Lambda_1 / \Lambda_0$. So we obtain an $\SO_X$-linear map
$$\psi_\nabla: F^*H \longrightarrow \End_{\SO_X}(E).$$
 
\begin{proposition}
\label{prop:p-curvature}
For each $\Lambda$-module $(E,\nabla)$ the $p$-curvature $\psi_\nabla$ induces a $F^*H^\vee$-valued Higgs field 
on $E$, i.e., a morphism of $\SO_X$-algebras
$$\tilde{\psi}_\nabla : \Sym^\bullet F^*H \longrightarrow \End_{\SO_X}(E).$$
Moreover, for any local sections $D\in H$ and  $D'\in \Lambda$, $\nabla_{D'}$ commutes with $\psi_\nabla(D)$.
\end{proposition}

\begin{proof}
We have already proven that the $p$-curvature induces an $\SO_X$-linear map 
$\psi_\nabla: F^*H \longrightarrow \End_{\SO_X}(E)$. In order for this map to lift to a morphism of 
algebras $\Sym^\bullet F^*H \longrightarrow \End_{\SO_X}(E)$, it is necessary that for each 
$D_1,D_2\in H$
$$\left [ \psi_\nabla(D_1), \psi_\nabla(D_2) \right] =0.$$
But, taking into account that from Proposition \ref{cor:iotaCenter} we know that the image 
of $\iota: F^*H \to \Lambda$ lies in the center $Z(\Lambda)$, we have
$$\left [ \psi_\nabla(D_1), \psi_\nabla(D_2) \right]= \left [ \nabla_{\iota(D_1)}, \nabla_{\iota(D_2)} \right] = \nabla_{[\iota(D_1),\iota(D_2)]}=\nabla_0=0.$$
The second part follows from a similar computation
$$\left [ \psi_\nabla(D), \nabla_{D'} \right]= \left [ \nabla_{\iota(D_1)}, \nabla_{D'} \right] = \nabla_{[\iota(D), D']}=\nabla_0=0.$$
\end{proof}

\begin{remark}
The previous proposition was already obtained in \cite[Lemma 4.9]{Langer} for modules over restricted
$\SO_S$-Lie algebroids $H$, which correspond to $\Lambda$-modules, where $\Lambda = \Lambda_H$ is the universal
enveloping algebra of the $\SO_S$-Lie algebroid $H$. We note that the proofs of the two previous propositions are similar to
those given in \cite{Langer}, but rely on the more general statement obtained in Proposition 
\ref{cor:iotaCenter} for general restricted sheaves of rings of differential operators.
\end{remark}

\bigskip

\section{Hitchin map for restricted $\Lambda$-modules}

In this section we assume that $X$ is an integral projective scheme over  $S = \mathrm{Spec}(\mathbb{K})$. 
This assumption is needed in our main
Theorem \ref{thm:HitchinDescent}. We refer the reader to \cite{Langer} sections 3.5 and 4.5 for a construction 
of the Hitchin map in the relative case.

\bigskip

Given a restricted sheaf $\Lambda$ of rings of differential operators on $X$ 
and a $\Lambda$-module $(E,\nabla)$ of rank $r$ over $X$, we 
have proved in Proposition \ref{prop:p-curvature} that the $p$-curvature of $(E, \nabla)$ 
induces a $F^*H^\vee$-valued Higgs field on $E$
$$ \psi_{\nabla} \in H^0(X, \End(E) \otimes F^* H^\vee). $$

\bigskip

Then, by taking the (classical) Hitchin map $h$ for rank-$r$ Higgs sheaves 
we obtain a point in the Hitchin base 
$\mathcal{A}_r(X,F^* H^\vee)$
$$h(E,\psi_\nabla)=(\tr(\wedge^k\psi_\nabla))_{k=1}^r \in 
\mathcal{A}_r(X,F^* H^\vee) := \bigoplus_{k=1}^r H^0(X,\Sym^k(F^*H^\vee)).$$
Therefore, the $p$-curvature map $(E, \nabla) \mapsto \psi_\nabla$ composed with the Hitchin map $h$ 
defines a map $h_\Lambda$
\begin{equation}
\label{eq:hitchinMapLambda}
h_\Lambda  : \SM^\Lambda_X(r,P) \longrightarrow \mathcal{A}_r(X,F^* H^\vee),  \quad \quad  (E,\nabla) \mapsto h(E, \psi_\nabla),
\end{equation}
where $\SM^\Lambda_X(r,P)$ denotes the coarse moduli space parameterizing Giesecker semi-stable $\Lambda$-modules over
$X$ of rank $r$ and with Hilbert polynomial $P$ (\cite{Simpson1}, \cite{Langer1}, \cite{Langer2}).
\bigskip

In order to understand the structure of the map $h_\Lambda$, let us first consider the example 
given by the trivial $p$-structure on the symmetric algebra $\Sym^\bullet(H)$ --- see  Subsection \ref{trivialpstructure}. 
In that case a $\Sym^\bullet(H)$-module is an $H^\vee$-valued Higgs sheaf and its $p$-curvature is just the $p$-th power of the Higgs field
$$\psi_\nabla(D) = \nabla_{D^p}=(\nabla_D)^p.$$
Then it is easily seen that the coefficients of the characteristic polynomial of $\psi_\nabla$ are 
pull-backs by the Frobenius map of global sections in $H^0(X,\Sym^k(H^\vee))$.

\bigskip


Before proving our main result on the map $h_\Lambda$, 
we will need to recall the definition of the canonical connection
$$ \nabla^{\op{can}} : F^* \mathcal{G} \longrightarrow F^* \mathcal{G} \otimes \Omega^1_X $$
on a pull-back sheaf $F^* \mathcal{G}$
for a coherent $\SO_X$-module $\mathcal{G}$ under the absolute Frobenius map $F$ of $X$. Over an affine open subset $\mathrm{Spec}(A) = U \subset X$, we denote the $A$-module of local sections $\mathcal{G}(U)$ by $M$. Then local sections 
of the pull-back $F^* \mathcal{G}(U)$ correspond to $A \otimes_A M$ with the $A$-module structure given by left 
multiplication and the action of $A$ on $A$ given by the Frobenius map $F$. In other words, we have the
identifications $\lambda^p a \otimes_A m = a \otimes_A \lambda m$ for any $\lambda,a \in A$ and $m \in M$. Then
with this notation the canonical connection is defined by
$$ \nabla^{\op{can}} : a \otimes_A m \mapsto da \otimes_A m,$$
or equivalently, $\nabla^{\op{can}}_{\partial}(a \otimes_A m) = \partial a \otimes m$ for any
vector field $\partial$.

\bigskip

\begin{lemma} \label{Frobdescentopen}
Let $\mathcal{G}$ be a torsion-free $\SO_X$-module over an integral scheme $X$ and let $s \in H^0(X, F^* \mathcal{G})$ be a global
section. Suppose that there exists an open subset $\Omega \subset X$ such that 
$$\nabla^{\op{can}}_\partial (s_{| \Omega}) = 0$$
for any vector field $\partial$ over $\Omega$. Then $s$ descends under the Frobenius map, i.e. there exists $s' 
\in H^0(X, \mathcal{G})$ such that $s = F^*(s')$.
\end{lemma}

\begin{proof}
It will be enough to show the statement locally on an affine open subset $\Spec(A)$ of $X$. We then apply Cartier's theorem
over $\Spec(K)$, where $K$ is the fraction field of $A$, and obtain the existence of the Frobenius descend $s'$ over $\Spec(K)$.
Now the section $s$ also descends over $\Spec(A)$ since $\mathcal{G}$ is torsion-free. The computations are straightforward 
and left to the reader.
\end{proof}

\begin{lemma}
\label{lemma:innerCommutation}
Let $(E, \nabla)$ be a $\Lambda$-module. Then for any local sections $D \in H = \Lambda_1 / \Lambda_0$ 
and $D' \in \Lambda_1$
we have the following commutative diagram
\begin{eqnarray*}
\xymatrixcolsep{5pc} 
\xymatrix{
E\otimes F^*H^\vee \ar[r]^-{\id\otimes D} \ar[d]_{\tilde{\nabla}_{D'}} & E \ar[d]^{\nabla_{D'}}\\
E\otimes F^*H^\vee \ar[r]^-{\id\otimes D} & E,
}
\end{eqnarray*}
where we define the endomorphism $\tilde{\nabla}_{D'}$ by
$$\tilde{\nabla}_{D'} = \nabla_{D'} \otimes \id_{F^*H^\vee} + \id_E \otimes \nabla_{\delta \circ \smb(D')}^{\op{can}}.$$
\end{lemma}

\begin{proof}
It is enough to work locally over an affine open subset $U = \mathrm{Spec}(A)$. 
Consider an irreducible tensor $v\otimes a \otimes h \in (E \otimes F^* H^\vee)(U)$, 
where $v \in E(U)$, $h\in H^\vee(U)$, $a \in A = \SO_X(U)$ and the last tensor 
product is taken over the Frobenius map, i.e.
$$\lambda^p v \otimes a \otimes h = v \otimes \lambda^p a \otimes h =  v \otimes a \otimes \lambda h.$$
Then, using associativity of the ring $\Lambda$ and the fact that $[D',f] = \delta_{\smb(D')}(f)$ 
for any $D' \in \Lambda_1$ and  $f \in \SO_X$, have the following
\begin{multline*}
\tilde{\nabla}_{D'}(v\otimes a \otimes h) = \nabla_{D'}(v)\otimes a \otimes h  + v \otimes 
 \nabla^{\op{can}}_{\delta \circ \smb(D')}(a \otimes h) \\
=\nabla_{D'}(v)\otimes a \otimes h + v \otimes \delta_{\smb(D')}(a) \otimes h.
\end{multline*}
Applying $\id \otimes D$ for a local section $D \in H$ we have
\begin{multline*}
(\id \otimes D) \circ \tilde{\nabla}_{D'}(v\otimes a \otimes h)  = 
\nabla_{D'}(v) \otimes a \otimes \langle h , D \rangle + v \otimes \delta_{\smb(D')}(a) \otimes \langle h, D \rangle 
\\
=\langle h , D \rangle^p a \nabla_{D'}(v) + \langle h, D \rangle^p \delta_{\smb(D')}(a) v,
\end{multline*}
where $\langle -,- \rangle$ denotes the standard pairing between $H^\vee$ and $H$. On the other hand
\begin{multline*}
\nabla_{D'} \circ (\id \circ D)(v \otimes a \otimes h ) = \nabla_{D'}(\langle h, D \rangle^p a v) = 
\langle h , D \rangle^p a \nabla_{D'}(v) + \delta_{\smb(D')}(\langle h, D \rangle^p a) v\\
=\langle h , D \rangle^p a \nabla_{D'}(v) + \langle h, D \rangle^p \delta_{\smb(D')}(a) v
\end{multline*}
so we obtain the desired equality for an irreducible tensor. By additivity we conclude equality for any 
local section of $E \otimes F^* H^\vee$.
\end{proof}

\begin{corollary}
\label{cor:innerFlatness}
Let $(E, \nabla)$ be a $\Lambda$-module and let $\psi_\nabla: E\longrightarrow E\otimes F^*H^\vee$ denote 
its $p$-curvature. Then for any local section $D'\in \Lambda_1$ the following diagram commutes
\begin{eqnarray*}
\xymatrixcolsep{5pc} 
\xymatrix{
E \ar[r]^-{\psi_\nabla} \ar[d]_{\nabla_{D'}} & E\otimes F^*H^\vee \ar[d]^{\tilde{\nabla}_{D'}}\\
E \ar[r]^-{\psi_\nabla} & E\otimes F^*H^\vee,
}
\end{eqnarray*}
where $\tilde{\nabla}_{D'}$ was defined in the previous lemma.
\end{corollary}

\begin{proof}
By Proposition \ref{prop:p-curvature} we know that for any local sections 
$D\in H$ and $D'\in \Lambda_1$ the two endomorphisms $\psi_\nabla(D)$ and $\nabla_{D'}$
commute. Moreover, $\psi_\nabla(D):E \to E$ is the composition of the following maps
$$E \stackrel{\psi_\nabla}{\longrightarrow} E\otimes F^*H^\vee \stackrel{\id\otimes D}{\longrightarrow} E$$
so we have the following diagram in which we know that the outer square and the 
inner right square (by Lemma \ref{lemma:innerCommutation}) are commutative
\begin{eqnarray*}
\xymatrixcolsep{5pc} 
\xymatrix{
E \ar[r]^-{\psi_\nabla} \ar@/^2pc/[rr]^-{\psi_\nabla(D)} \ar[d]_{\nabla_{D'}} & E\otimes F^*H^\vee \ar[r]^-{\id\otimes D} \ar[d]^{\tilde{\nabla}_{D'}} & E \ar[d]^{\nabla_{D'}}\\
E \ar[r]^-{\psi_\nabla} \ar@/_1pc/[rr]_-{\psi_\nabla(D)} & E\otimes F^*H^\vee \ar[r]^-{\id\otimes D} & E.
}
\end{eqnarray*}
Thus, for any $D\in H$ and $D'\in \Lambda_1$
\begin{multline*}
0=\nabla_{D'} \circ \psi_\nabla(D) - \psi_\nabla(D)\circ \nabla_{D'} = \nabla_{D'}\circ (\id\otimes D) \circ \psi_\nabla - (\id\otimes D) \circ \psi_\nabla \circ \nabla_{D'}\\
=(\id \otimes D) \circ \tilde{\nabla}_{D'} \circ \psi_\nabla - (\id \otimes D) \circ \psi_\nabla \circ \nabla_{D'}\\
=(\id \otimes D) \circ \left(\tilde{\nabla}_{D'} \circ \psi_\nabla - \psi_\nabla \circ \nabla_{D'} \right).
\end{multline*}
As this composition is zero for any $D\in H$ and the kernel of the evaluation map in $F^*H^\vee$ is trivial, we obtain that
$$\tilde{\nabla}_{D'} \circ \psi_\nabla - \psi_\nabla \circ \nabla_{D'}=0.$$
\end{proof}


The next proposition will be used in the proof of the main result (Theorem \ref{thm:HitchinDescent}).

\begin{proposition} \label{pcurvisflat}
Assume that $p = \mathrm{char}(\mathbb{K}) > 2$. 
Let $\Lambda$ be a restricted sheaf of differential operators on $X$ over $S$ and let $\mathcal{E}$
be a coherent $\SO_X$-module together with a morphism of $\SO_S$-modules
$\nabla : \Lambda_1 \longrightarrow \End_{\SO_S}(\mathcal{E})$ satisfying for any local sections $f,g 
\in \SO_X$, $s \in \mathcal{E}$ and $D \in \Lambda_1$
$$\nabla_{D}(fs) = f \nabla_D(s) + \delta_{\smb(D)} \quad \text{and} \quad \nabla_g (s) = gs.$$
Let $\mathcal{G}$ be a coherent $\SO_X$-module and let $\psi : \mathcal{E} \to 
\mathcal{E} \otimes F^* \mathcal{G}$ be an $\SO_X$-linear map. Suppose that for $D \in \Lambda_1$ we have 
a commutative diagram 
\begin{eqnarray*}
\xymatrixcolsep{5pc} 
\xymatrix{
\mathcal{E} \ar[r]^-{\psi} \ar[d]_{\nabla_{D}} & \mathcal{E} \otimes F^*\mathcal{G} 
\ar[d]^{\tilde{\nabla}_{D}}\\
\mathcal{E} \ar[r]^-{\psi} & \mathcal{E} \otimes F^* \mathcal{G},
}
\end{eqnarray*}
where the endomorphism $\tilde{\nabla}_{D}$ on the right is defined by
$$\tilde{\nabla}_{D} = \nabla_{D} \otimes \id_{F^* \mathcal{G}} + \id_{\mathcal{E}} 
\otimes \nabla_{\delta \circ \smb(D)}^{\op{can}}.$$
Then over an open dense subset of $X$ we have 
$$\nabla^{\op{can}}_{\delta \circ \smb(D)} (\mathrm{tr}(\psi)) = 0,$$
where $\mathrm{tr}(\psi) \in H^0(X, F^* \mathcal{G})$ denotes the trace of the $\SO_X$-linear 
endomorphism $\psi$.
\end{proposition}

\begin{proof}
Since $X$ is integral, we can restrict attention to the open dense subset $\Omega \subset X$ where both
$\mathcal{E}$ and $\mathcal{G}$ are locally free. Moreover, it will be enough to check the equality 
locally. For $x \in \Omega$ we denote by $\mathcal{O}$ the local ring of $\mathcal{O}_X$ at the point $x$.
Then we can write
$$ \nabla_D = \partial + A, $$
where $\partial =  \delta \circ \smb(D)$ is a derivation on $\mathcal{O}$ and $A$ is a 
$r \times r$ matrix with values in $\mathcal{O}$ and $r = \mathrm{rk}(\mathcal{E})$. Similarly, let $n
= \mathrm{rk}(\mathcal{G})$ and choosing an $\mathcal{O}$-basis of $\mathcal{G}_x$ 
then $\psi$ corresponds to $n$ $r \times r$ matrices
$B_1, \dots, B_n$ with values in $\mathcal{O}$. Then the commutation relations translate into
the following $n$ equalities for $i = 1, \ldots, n$ in 
$\End(\mathcal{O}^{\oplus r})$
\begin{eqnarray*}
B_i (\partial + A) & = & ((\partial + A) \otimes \id + \id \otimes \partial) B_i \\
 \Longleftrightarrow  \ \  B_i \partial + B_i A & = & \partial B_i + AB_i + \partial. B_i \\
 \Longleftrightarrow  \ \  [B_i, A] & = & [\partial, B_i] + \partial. B_i = 2 \partial. B_i,
\end{eqnarray*}
where $\partial. B_i$ denotes the matrix obtained from $B_i$ by applying the derivation 
$\partial$ to all of its coefficients. We also have used the well-known identity 
$[\partial, B_i] = \partial. B_i$. Taking the trace, we obtain
$$ 0 = \mathrm{tr}([B_i, A]) = 2  \mathrm{tr} ( \partial. B_i ) = 2 \partial ( \mathrm{tr}(B_i)). $$
Hence, since $p>2$, we obtain 
$$ \partial ( \mathrm{tr}(B_i)) = 0 \quad  \text{for} \  i = 1, \ldots, n. $$
This shows the result.
\end{proof}

\begin{proposition} \label{pcurvcommutes}
Let $(E,\nabla)$ be a restricted $\Lambda$-module of rank $r$ and let $\psi_\nabla : E \to 
E \otimes F^* H^\vee$ denote its $p$-curvature. Then for $i=1,\ldots, r$ the $\SO_X$-linear 
composite map
$$ \psi_i : \Lambda^i E \stackrel{\Lambda^i \psi_\nabla}{\longrightarrow}  
\Lambda^i ( E \otimes F^* H^\vee) \stackrel{pr}{\longrightarrow}   \Lambda^i E \otimes  F^* \Sym^i H^\vee $$
of $\Lambda^i \psi_\nabla$ with the natural projection map $\mathrm{pr}$ satisfies the 
commutation relations of Proposition \ref{pcurvisflat} with 
$\mathcal{E} = \Lambda^i E$, $\mathcal{G} = \Sym^i H^\vee$, $\psi = \psi_i$ and the natural actions of 
$\Lambda_1$ on $\mathcal{E}$ and $\mathcal{G}$ induced by $\nabla$.
\end{proposition}

\begin{proof}
We observe that if $(E, \nabla)$ is a $\Lambda$-module, the exterior power $\Lambda^i E$ need not
necessarily be a $\Lambda$-module, but $\Lambda^i E$ can be equipped by an action of $\Lambda_1$
satisfying the properties given in Proposition \ref{pcurvisflat}. Since $\psi_i$ is a composite map, it
will be enough to check that the two maps $\Lambda^i \psi_\nabla$ and $\mathrm{pr}$ satisfy the
commutation relations. Both checks follow immediately from the definitions of the maps.
\end{proof}

\bigskip
We can now state our main result.
\bigskip

\begin{theorem}
\label{thm:HitchinDescent}
Assume that $p = \mathrm{char}(\mathbb{K}) > 2$. 
Let $\Lambda$ be a restricted sheaf of rings of differential operators over $X$.
We assume that $H = \Lambda_1/\Lambda_0$ is locally free and that the anchor map $\delta : H  \rightarrow T_X$ is generically surjective. 
Then the coefficients $\tr(\psi_i)$ of the characteristic polynomial of 
the $p$-curvature $\psi_\nabla$ of a $\Lambda$-module $(E,\nabla)$ over $X$ are $p$-th powers, i.e. descend under the
Frobenius map $F$ of $X$. This implies that the above defined Hitchin map $h_\Lambda$ \eqref{eq:hitchinMapLambda} factorizes
as follows
\begin{eqnarray*}
\xymatrix{
\SM^\Lambda_X(r,P) \ar[r]^{h_\Lambda} \ar[dr]^{h'_\Lambda} & \mathcal{A}_r(X, F^* H^\vee) \\
   & \mathcal{A}_r(X, H^\vee) \ar@{^(->}[u]^{F},
}
\end{eqnarray*}
where the vertical map is the pull-back map of global sections under the Frobenius map $F$ of $X$.
\end{theorem}

\begin{proof}
Let $(E, \nabla)$ be a restricted $\Lambda$-module of rank $r$ with $p$-curvature $\psi_\nabla$.
Proposition \ref{pcurvcommutes} shows that the global section $\psi_i : \Lambda^i E \rightarrow \Lambda^i E \otimes F^*\Sym^i H^\vee$
obtained by projecting $\Lambda^i \psi_\nabla$ satisfies the commutation relations of Proposition \ref{pcurvisflat}. Therefore,
applying Proposition \ref{pcurvisflat},  we can conclude that for any local section $D \in \Lambda_1$
$$ \nabla^{\op{can}}_{\delta \circ \smb(D)} (\mathrm{tr}(\psi_i)) = 0 $$
over an open subset $\Omega_1$ of $X$. Let $\Omega_2$ be an open subset where the anchor map $\delta$ is surjective. Then
over $\Omega_1 \cap \Omega_2$ we have $\nabla^{\op{can}}_{\partial} (\mathrm{tr}(\psi_i)) = 0$ for any local vector field
$\partial$. Now we can apply Lemma \ref{Frobdescentopen}, since $X$ is integral and $H$ locally free.
\end{proof}

\begin{remark}
The following example shows that the assumption that $\delta: H \rightarrow T_X$ is generically surjective cannot be dropped
in Theorem \ref{thm:HitchinDescent}.
Let $X$ be a smooth projective curve of genus $g \geq 2$ over $S = \mathrm{Spec}(\mathbb{K})$ and 
let $T_X$ (resp. $K_X$)  be its tangent (resp. canonical) line bundle. We choose a non-zero global section 
$\varphi \in H^0(X, K_X^{p-1})$ with reduced zero divisor.
We consider as explained in Subsection \ref{pstrsymmalg} the symmetric 
algebra $\Lambda = \mathrm{Sym}^\bullet (T_X)$ with the $p$-structure given by the $\SO_X$-linear map
$\alpha : F^* T_X = T_X^{\otimes p} \rightarrow T_X$ corresponding to the multiplication with $\varphi$.
Note that in this case $\delta = 0$. Then a $\Lambda$-module $(E, \nabla)$ over $X$ corresponds to a vector bundle $E$ together with a Higgs field, i.e., an
$\SO_X$-linear map $\nabla : E \rightarrow E \otimes K_X$. The $p$-curvature $\psi_{\nabla}$ of $(E, \nabla)$
then correponds to the $\SO_X$-linear map $E \rightarrow E \otimes F^* K_X$
$$ \psi_{\nabla} = \nabla^p - \alpha. \nabla,$$
where $\alpha . \nabla$ denotes the composite map $(\mathrm{id}_{E} \otimes \alpha^\vee) \circ \nabla$. Then clearly
$\mathrm{tr}(\psi_\nabla)$ does not descend under the Frobenius map.
\end{remark}

\begin{remark}
The previous remark shows that asking for a generally surjective anchor $\delta:H\longrightarrow T_X$ is indeed necessary for the Theorem, but it can be proven that, in some scenarios, this condition is generically satisfied. For instance, if $X$ is a smooth curve, then any nonzero map $\delta:H\longrightarrow T_X$ is generically surjective. As a consequence, for smooth curves, any restricted sheaf of rings of differential operators $\Lambda$ on $X$ with nonzero anchor satisfies Theorem \ref{thm:HitchinDescent}. In particular, this holds when $\Lambda$ is the universal enveloping algebra of any restricted Lie algebroid $(H,[-.-],\delta,[p])$ with $\delta \ne 0$ or, more generally, for any $\Lambda$ in which the left and right $\SO_X$-module structures are different (see Remark \ref{rmk:delta0}).
\end{remark}

\section{Hitchin map for restricted $\Lambda^R$-modules}

The argument used to show Theorem \ref{thm:HitchinDescent} can be adapted to the following particular relative case: consider
for an integral projective scheme $X$ over $S = \mathrm{Spec}(\mathbb{K})$ the restricted sheaf of rings of differential 
operators $\Lambda^R$ over $X \times \mathbb{A}^1$ relative to $\mathbb{A}^1$ obtained via the Rees construction from
the universal enveloping algebra $\Lambda = \Lambda_H$ of a restricted Lie algebroid $(H, [-,-], \delta, [p])$ 
over $X$ --- see Subsection \ref{redgrad}.

\bigskip

We consider the moduli space $\mathcal{M}^{\Lambda^R}_{X \times \mathbb{A}^1 / \mathbb{A}^1}(r,P)$
parameterizing Giesecker semi-stable $\Lambda^R$-modules over $X \times \mathbb{A}^1 / \mathbb{A}^1$ of rank $r$
and with Hilbert polynomial $P$. Since $\Lambda_1^R /\Lambda_0^R = p_X^*(H)$ the Hitchin map $h_{\Lambda^R}$ in the relative case 
(see \cite[Sections 3.5 and 4.5]{Langer}) corresponds to a morphism
\begin{eqnarray} \label{HitchinmapR}
\xymatrix{
\mathcal{M}^{\Lambda^R}_{X \times \mathbb{A}^1 / \mathbb{A}^1}(r,P) \ar[r]^{h_{\Lambda^R}} \ar[dr]  & \mathcal{A}_r(X, F^* H^\vee) \times 
\mathbb{A}^1 \ar[d]^{p_{\mathbb{A}^1}} \\
   & \mathbb{A}^1 
}
\end{eqnarray}
over $\mathbb{A}^1$. Then we obtain the 

\begin{theorem}    \label{HitchindescentRees}
Assume that $p = \mathrm{char}(\mathbb{K}) > 2$. 
Let $\Lambda = \Lambda_H$ be the universal enveloping algebra of a restricted Lie algebroid $(H, [-,-], \delta, [p])$ 
over an integral projective scheme $X$ and let $\Lambda^R$ be  the restricted sheaf of rings of differential 
operators over $X \times \mathbb{A}^1$ relative to $\mathbb{A}^1$ obtained via the Rees construction from $\Lambda_H$.
We assume that $H = \Lambda_1/\Lambda_0$ is locally free and that the anchor map $\delta : H  \rightarrow T_X$ is generically surjective. 
Then the above defined Hitchin map $h_{\Lambda^R}$ \eqref{HitchinmapR} factorizes as follows
\begin{eqnarray*}
\xymatrix{
\mathcal{M}^{\Lambda^R}_{X \times \mathbb{A}^1 / \mathbb{A}^1}(r,P)  \ar[r]^{h_{\Lambda^R}} \ar[dr]^{h'_{\Lambda^R}} & 
\mathcal{A}_r(X, F^* H^\vee) \times \mathbb{A}^1 \\
   & \mathcal{A}_r(X, H^\vee)  \times \mathbb{A}^1, \ar@{^(->}[u]^{F \times \mathrm{id}}
}
\end{eqnarray*}
where the vertical map is the pull-back map of global sections under the Frobenius map $F$ of $X$.

\end{theorem}

\begin{proof}
Since the anchor $\delta : H \rightarrow T_X$ is generically surjective over $X$, the anchor $\delta^R = t \delta : p_X^* (H) \rightarrow p_X^*(T_X)$
is also generically surjective over $X \times \mathbb{A}^1$. Hence we can apply the same arguments as in the proof of Theorem \ref{thm:HitchinDescent}
for local relative vector fields $\partial \in T_{X \times \mathbb{A}^1/ \mathbb{A}^1} = p_X^*(T_X)$.
\end{proof}

\bibliographystyle{alpha}
\bibliography{Biblio}

\end{document}